\documentclass[10pt]{article}
\usepackage{latexsym,amssymb,amsmath,amsthm,amsfonts}
\usepackage{graphicx,color}
\usepackage{hyperref}

\theoremstyle{plain}
\newtheorem{thm}{Theorem}[section]

\newtheorem{lem}[thm]{Lemma}

\theoremstyle{definition}
\newtheorem{definition}[thm]{Definition}
\theoremstyle{remark}

\numberwithin{equation}{section}

\newcommand{\RR}{\mathbb{R}}
\newcommand{\f}{\frac}

\newcommand{\xx}{|x|^2}
\newcommand{\yy}{|y|^2}
\newcommand{\xy}{\langle x,y\rangle}

\newcommand{\ppp}[3]{\frac{\partial^2{#1}}{\partial{#2}\partial{#3}}}
\newcommand{\pppp}[4]%
  {\frac{\partial^3{#1}}{\partial{#2}\partial{#3}\partial{#4}}}

\newcommand{\p}{\phi}

\renewcommand{\a}{\alpha}
\renewcommand{\b}{\beta}
\newcommand{\ab}{(\alpha,\beta)}
\newcommand{\ta}{\tilde\alpha}
\newcommand{\tb}{\tilde\beta}
\newcommand{\ha}{\hat\alpha}
\newcommand{\hb}{\hat\beta}
\newcommand{\ba}{\bar\alpha}
\newcommand{\bb}{\bar\beta}

\newcommand{\aij}{a_{ij}}

\newcommand{\bi}{b_i}
\newcommand{\bj}{b_j}
\newcommand{\bk}{b_k}

\newcommand{\hbi}{\hat b_i}
\newcommand{\hbj}{\hat b_j}

\newcommand{\bbi}{\bar b_i}
\newcommand{\bbj}{\bar b_j}

\newcommand{\bij}{b_{i|j}}
\newcommand{\taij}{\tilde a_{ij}}

\newcommand{\tbij}{\tilde b_{i|j}}
\newcommand{\haij}{\hat a_{ij}}

\newcommand{\hbij}{\hat b_{i|j}}
\newcommand{\baij}{\bar a_{ij}}

\newcommand{\bbij}{\bar b_{i|j}}

\newcommand{\G}{G^i_\alpha}
\newcommand{\tG}{\tilde G^i_{\tilde\alpha}}
\newcommand{\hG}{\hat G^i_{\hat\alpha}}
\newcommand{\bG}{\bar G^i_{\bar\alpha}}

\newcommand{\rij}{r_{ij}}
\newcommand{\sij}{s_{ij}}
\newcommand{\ri}{r_i}
\newcommand{\si}{s_i}
\newcommand{\sio}{s^i{}_0}
\newcommand{\rj}{r_j}
\newcommand{\sj}{s_j}
\newcommand{\trij}{\tilde r_{ij}}
\newcommand{\hrij}{\hat r_{ij}}
\newcommand{\brij}{\bar r_{ij}}
\newcommand{\roo}{r_{00}}

\begin{document}
\title{On dually flat Randers metrics}
\author{Changtao Yu}
\date{}
\maketitle

\begin{abstract}
The notion of dually flat Finsler metrics arise from information geometry. In this paper, we will study a special class of Finsler metrics called Randers metrics to be dually flat. A simple characterization is provided and some non-trivial explicit examples are constructed. In particular, We will show that the dual flatness of a Randers metric always arises from that of some Riemannian metric by doing some special deformations.
\end{abstract}

\section{Introduction}
The notion of dual flatness for metrics was first proposed by S.-I. Amari and H. Nagaoka when they studied the information geometry on Riemannian manifold\cite{AN}. Later on, Z. Shen studied the information geometry in Finsler geometry and introduced the notion of dually flat Finsler metrics\cite{szm-rfgw}. A Finsler metric on a manifold is said to be {\it locally dually flat} if at any point there is a local coordinate system in which the spray coefficients of $F$ are in the form
\begin{eqnarray*}
G^i=-\f{1}{2}g^{ij}H_{y^j},
\end{eqnarray*}
where $H=H(x,y)$ is a scalar function on the tangent bundle $TM$. Z. Shen's result says that A Finsler metric $F(x,y)$ on an open subset $U\subseteq\RR^n$ is dually flat if and only if the following PDEs hold\cite{szm-rfgw}:
\begin{eqnarray*}
[F^2]_{x^ky^l}-2[F^2]_{x^l}=0.
\end{eqnarray*}

For a Riemannian metric $F=\sqrt{g_{ij}(x)y^iy^j}$, it is known that it is dually flat on $U$ if and only if its fundamental tensor is the Hessian of some local smooth function $\psi(x)$\cite{AN}, i.e.,
\begin{eqnarray*}
g_{ij}(x)=\ppp{\psi}{x^i}{x^j}(x).
\end{eqnarray*}

The first example of non-Riemannian dually flat Finsler metrics is the co-call {\it Funk metric}
\begin{eqnarray*}
F=\f{\sqrt{(1-\xx)\yy+\xy^2}}{1-\xx}\pm\f{\xy}{1-\xx}
\end{eqnarray*}
on the unit ball $\mathbb B^n(1)$\cite{csz-oldf}.

Funk metric belongs to a special class of Finsler metrics named {\it Randers metrics}. They are expressed as $F=\a+\b$, where $\a=\sqrt{a_{ij}(x)y^iy^j}$ is a Riemannian metric and $\b=b_i(x)y^i$ is an $1$-form with $b:=\|\b\|_\a<1$. Randers metrics were first introduced by a physicist G. Randers in 1941 when he studied general relativity. This special class of Finsler metrics play an important role in the research on Finsler geometry partly because of its computability. Many inspirational results have been obtained. For instance, based on the {\it navigation problem} on Riemannian manifold, D. Bao et al. classified the Randers metrics of constant flag curvature\cite{brs}. Here flag curvature is the most important geometrical quantity for Finsler metrics, which is the extension of sectional curvature in Riemannian geometry.

According to navigation problem, a Randers metric $F=\a+\b$ can always be expressed as
\begin{eqnarray*}\label{NPexpression}
F=\f{\sqrt{(1-|W|^2_h)h^2+(W^\flat)^2}}{1-|W|_h^2}-\f{W^\flat}{1-|W|^2_h},
\end{eqnarray*}

\noindent where $h$ is a Riemannian metric on the manifold $M$ and $W$ is a vector field with $|W|_h<1$. Here $W^\flat$ means the dual $1$-form of $W$ with respect to $h$. The {\it navigation data} $(h,W)$ is related to the original data $\ab$ by
\begin{eqnarray*}
h=\sqrt{1-b^2}\sqrt{\a^2-\b^2},\quad W^\flat=-(1-b^2)\b.
\end{eqnarray*}

\noindent The classification result shows that {\it a Randers metric $F=\a+\b$ is of constant flag curvature if and only if $h$ is of constant sectional curvature and $W$ is an infinitesimal homothety of $h$}\cite{brs}. For example, the Funk metric is of constant flag curvature $-\frac{1}{4}$ and all of its geodesics are straight lines.

Although it is known that the Funk metric is dual flat, it is almost the only known one in the class of dually flat Randers metrics\cite{csz-oldf}. The aim of this paper is to provide a more direct characterization of the dually flat Randers metrics based on the following result:
\begin{thm}\cite{csz-oldf}\label{maincf}
Let $F=\a+\b$ be a Randers metric on an open subset $U\subseteq\RR^n$. Then $F$ is dually flat if and only if $\a$ and $\b$ satisfy
\begin{eqnarray}
\G&=&(2\theta+\tau\b)y^i-\a^2(\tau b^i-\theta^i),\label{G}\\
\roo&=&2\theta\b-5\tau\b^2+(3\tau+2\tau b^2-2b_k\theta^k)\a^2,\label{rij}\\
s_{i0}&=&\b\theta_i-\theta b_i,\label{sij}
\end{eqnarray}
where $\theta=\theta_k(x)y^k$ is an $1$-form on $U$, $\theta^i:=a^{ik}\theta_k$, and $\tau=\tau(x)$ is a scalar function.
\end{thm}

More specifically we will prove
\begin{thm}\label{main1}
Let $F=\a+\b$ be a Randers metric on an open subset $U\subseteq\RR^n$. Then the following items are equivalent:
\begin{enumerate}
\item $F$ is dually flat on $U$;
\item The navigation data $(h,W)$ of $F$ satisfies
\begin{eqnarray}\label{xx}
G^i_h=2\xi y^i+h^2\xi^i,\quad W_{i|j}=c(x)h_{ij}+2\xi_iW_i,
\end{eqnarray}
where $\xi:=\xi_iy^i$ is an $1$-form on $U$, $\xi^i:=h^{ij}\xi_j$, $c(x)$ is a scalar function;
\item The Riemannian metric $\ba:=(1-b^2)^{\frac{1}{4}}\a$ and the $1$-form $\bb:=(1-b^2)^{-\frac{1}{4}}\b$ satisfy
\begin{eqnarray}\label{yy}
\bG=2\bar\theta y^i+\ba^2\bar\theta^i,\quad\bbij=\bar c(x)\baij+2\bar\theta_i\bbj,
\end{eqnarray}
where $\bar\theta:=\bar\theta_iy^i$ is an $1$-form on $U$, $\bar\theta^i:=\bar a^{ij}\bar\theta_j$, $\bar c(x)$ is a scalar function.
\end{enumerate}
\end{thm}

Note that the notations $W_{i|j}$ and $\bbij$ mean the covariant derivation of $W$ and $\bb$ with respect to $h$ and $\ba$ respectively.

The conditions for the Riemannian metrics in (\ref{xx}) and (\ref{yy}) are both imply that the corresponding Riemannian metric is dually flat on $U$ (see Section \ref{s3} for the reason). In other words, {\it the dual flatness of a Randers metric always arises from that of some Riemannian metric}.

Projectively flat $\ab$-metrics have the similar phenomenon. Recall that a Finsler metric on $U$ is said to {\it projectively flat} if all of its geodesics are straight lines, such as the Funk metric. {\it$\ab$-metrics}, which become a more extensive class of computable Finsler metrics\cite{bcs} including Randers metrics naturally, are defined by a Riemmannian metric and an $1$-form. The author proved in the doctoral dissertation that {\it the projective flatness of a non-trivial $\ab$-metric on a manifold with dimension $n\geq3$ always arises from that of some Riemannian metric}\cite{yct-dhfp}.

On the other hand, from Theorem \ref{main1} we can see that a special kind of $1$-forms appear frequently (see Section \ref{s3} and Section \ref{s4} for the related argument). It should be pointed out that the properties of such kind of $1$-forms are still not clear enough, but it seems that it is important for dually flat Finsler metrics. Actually, this kind of $1$-forms are also the main key in the discussion of dually flat $\ab$-metrics\cite{xoldf,yct-odfa}.

A special kinds of metric deformations called {\it $\b$-deformations} are the main tools in our discussions. They are first proposed by the author in his research on the projectively flat $\ab$-metrics\cite{yct-dhfp}, and they are effective for many other problems\cite{szm-yct-oesm,yct-odfa,yct-zhm-onan}. Actually, the navigation technique (\ref{NPexpression}) is just a special kind of $\b$-deformations. So one can regard $\b$-deformations as the generalization of the navigation technique for Randers metrics.

Also by using this new deformation method, we find some Riemannian metrics and $1$-forms satisfying the conditions in Theorem \ref{main1}, and hence construct some interesting dually flat Randers metrics below.
\begin{thm}\label{main2}
The following Randers metrics
\begin{eqnarray}\label{example}
F(x,y)=\f{\sqrt[4]{1+(\mu+\lambda^2)\xx}\sqrt{(1+\mu\xx)\yy-\mu\xy^2}}{1+\mu|x|^2}
+\f{\lambda\xy}{(1+\mu\xx)\sqrt[4]{1+(\mu+\lambda^2)\xx}}
\end{eqnarray}
are dually flat on $\mathbb B^n(r_\mu)$, where $\mu$ and $\lambda$ are constants, and the radius $r_\mu$ is determined by $r_\mu:=\frac{1}{\sqrt{-\mu}}$ when $\mu<0$ and $r_\mu:=+\infty$ when $\mu\geq0$.
\end{thm}

Taking $\mu=-1$ and $\lambda=\pm1$ in (\ref{example}) we get the Funk metric again. Taking $\mu=0$ and $\lambda=\pm1$ we obtain a simper metric in its form as following:
$$F(x,y)=(1+\xx)^\frac{1}{4}|y|\pm(1+\xx)^{-\frac{1}{4}}\xy.$$

All the examples provided above are non-trivial. See Section \ref{s4} for the reason.

\section{Preliminaries}
Let $M$ be a smooth $n$-dimensional manifold. A Finsler metric $F$ on $M$ is a continuous function
$F:TM\to[0,+\infty)$ with the following properties:
\begin{enumerate}
\item {\it Regularity}: $F$ is $C^\infty$ on the entire slit tangent bundle $TM\backslash\{0\}$;
\item {\it Positive homogeneity}: $F(x,\lambda y)=\lambda F(x,y)$ for all $\lambda>0$;
\item {\it Strong convexity}: the fundamental tensor $g_{ij}:=[\frac{1}{2}F^2]_{y^iy^j}$ is positive definite for all $(x,y)\in TM\backslash\{0\}$.
\end{enumerate}
Here $x=(x^i)$ and $y=(y^i)$ denote the coordinates of the points in $M$ and the vectors in $T_xM$ respectively.

Given a Finsler metric $F$ on $M$. There is a global vector field $G$ on $TM\backslash\{0\}$ called  a {\it spray}. In local coordinates, $G=y^i\frac{\partial}{\partial x^i}-2G^i\frac{\partial}{\partial y^i}$ where
$$G^i:=\f{1}{4}g^{il}\left\{[F^2]_{x^ky^l}y^k-[F^2]_{x^l}\right\}$$
are called the {\it spray coefficients} of $F$. Here $(g^{ij})$ is the inverse of $(g_{ij})$. For a Riemannian metric, the spray coefficients are determined by the Christoffel symbols as
$$G^i(x,y)=\frac{1}{2}\Gamma^i{}_{jk}(x)y^jy^k.$$

In the rest of this section, we introduce some basic notions of $\b$-deformations. Given a Riemannian metric $\a$ and an $1$-form $\b$, set $b:=\|\b\|_\a$. By definition, the $\b$-deformations are a triple of metric deformations in terms of $\a$ and $\b$ listed below:
\begin{eqnarray*}
&\ta=\sqrt{\a^2-\kappa(b^2)\b^2},\qquad\tb=\b;\\
&\ha=e^{\rho(b^2)}\ta,\qquad\hb=\tb;\\
&\ba=\ha,\qquad\bb=\nu(b^2)\hb.
\end{eqnarray*}
Here we choose $b^2$ instead of $b$ as the variable, because it will be convenient for computations. Notice that in order to keep the positive definition of $\ta$, $\kappa(b^2)$ must satisfies an additional condition:
\begin{eqnarray}\label{conditiononk}
1-\kappa b^2>0.
\end{eqnarray}

Obviously, the first two kinds of $\b$-deformations aim for Riemannian metrics and the last one is for $1$-forms. More specifically, the first kind of $\b$-deformation can be regarded as some kind of stretch change for $\a$ along the direction determined by $\b$, the second one is conformal change and the third one is just the length change of $\b$. The main feature of such special kinds of derormations, which makes the whole discussion concise, is that all the deformation factors are functions of $b$ instead of functions of points. .

Some basic formulas are listed below. It should be attention that the notation `$\dot b_{i|j}$' always means the covariant derivative of the $1$-form `$\dot\b$' with respect to the corresponding Riemannian metric `$\dot\a$', where the symbol `~$\dot{}$~' can be `~$\tilde{}$~', `~$\hat{}$~' or `~$\bar{}$~' in this paper. Moreover, we need the following abbreviations,
\begin{eqnarray*}
r_{00}:=r_{ij}y^iy^j,~r_i:=r_{ij}y^j,~r_0:=r_iy^i,~r:=r_ib^i,
~s_{i0}:=s_{ij}y^j,~s^i{}_0:=a^{ij}s_{j0},~s_i:=s_{ij}y^j,~s_0:=s_ib^i,
\end{eqnarray*}
where $\rij$ and $\sij$ are the symmetrization and antisymmetrization of $\bij$ respectively, i.e.,
$$\rij:=\f{1}{2}(\bij+b_{j|i}),\quad\sij:=\f{1}{2}(\bij-b_{j|i}).$$
It is clear that $\sij=0$ if and only if $\b$ is closed.

\begin{lem}\cite{yct-dhfp}\label{beta1}
Let $\ta=\sqrt{\a^2-\kappa(b^2)\b^2}$, $\tb=\b$. Then
\begin{eqnarray*}
\tG&=&\G-\frac{\kappa}{2(1-\kappa b^2)}\big\{2(1-\kappa b^2)\b s^i{}_0+r_{00}b^i+2\kappa s_0\b b^i\big\}\\
&&+\frac{\kappa'}{2(1-\kappa b^2)}\big\{(1-\kappa b^2)\b^2(r^i+s^i)+\kappa r\b^2b^i-2(r_0+s_0)\b
b^i\big\},\\
\tbij7&=&\bij+\frac{\kappa}{1-\kappa b^2}\big\{b^2\rij+\bi\sj+\bj\si\big\}
-\frac{\kappa'}{1-\kappa b^2}\big\{r\bi\bj-b^2\bi(\rj+\sj)-b^2\bj(\ri+\si)\big\}.
\end{eqnarray*}
\end{lem}

\begin{lem}\cite{yct-dhfp}\label{beta2}
Let $\ha=e^{\rho(b^2)}\ta$, $\hb=\tb$. Then
\begin{eqnarray*}
\hG&=&\tG+\rho'\left\{2(r_0+s_0)y^i-(\a^2-\kappa \b^2)\left(r^i+s^i+\frac{\kappa}{1-\kappa b^2}rb^i\right)\right\},\\
\hbij&=&\tbij-2\rho'\left\{\bi(\rj+\sj)+\bj(\ri+\si)-\frac{1}{1-\kappa b^2}r(\aij-\kappa\bi\bj)\right\}.
\end{eqnarray*}
\end{lem}

\begin{lem}\cite{yct-dhfp}\label{beta3}
Let $\ba=\ha$, $\bb=\nu(b^2)\hb$. Then
\begin{eqnarray*}
\bG=\hG,\qquad\bbij=\nu\hbij+2\nu'\bi(\rj+\sj).
\end{eqnarray*}
\end{lem}

\section{Proof of Theorem \ref{main1}}\label{s3}
In this section, we will simplify those $\a$ and $\b$ who satisfy (\ref{G})-(\ref{sij}) by $\b$-deformations. Our first aim is to make $\a$ dually flat, and then we will seek a suitable standard to simplify $\b$.

Firstly, the dual flatness of a Riemannian metric can be described as following, which is a obvious corollary of Theorem \ref{maincf}.
\begin{lem}\label{Rd}
Let $\a$ be a Riemannian metric on an open subset $U\subseteq\RR^n$. Then $\a$ is dually flat if and only if exists an $1$-form $\theta$ on $U$ such that
\begin{eqnarray}\label{duallyflat}
\G=2\theta y^i+\a^2\theta^i.
\end{eqnarray}
\end{lem}

Suppose that $F=\a+\b$ is a dually flat Randers metric on $U$. According Theorem \ref{maincf}, it is easy to obtain the following simple facts:
\begin{eqnarray}
\rij&=&\theta_i\bj+\theta_j\bi-5\tau\bi\bj+(3\tau+2\tau b^2-2b_k\theta^k)\aij,\label{temp1}\\
\sio&=&\b\theta^i-\theta b^i,\label{temp2}\\
s_0&=&b_k\theta^k\b-b^2\theta,\label{temp3}\\
\ri+\si&=&3\tau(1-b^2)\bi,\label{temp4}\\
\bi\sj+\bj\si&=&2\bk\theta^k\bi\bj-b^2(\theta_i\bj+\theta_j\bi),\label{temp5}\\
r&=&3\tau(1-b^2)b^2.\label{temp6}
\end{eqnarray}

Carry out the first step of $\b$-deformations, then by (\ref{G}), (\ref{rij}), (\ref{temp1})-(\ref{temp6}) and Lemma \ref{beta1} we obtain
\begin{eqnarray*}
\tG&=&(2\theta+\tau\b)y^i-(\tau b^i-\theta^i)\a^2-\f{\kappa}{2(1-\kappa b^2)}\big\{2(1-\kappa b^2)\b(\b\theta^i-\theta b^i)+2\theta\b b^i-5\tau\b^2b^i\nonumber\\
&&+(3\tau+2\tau b^2-2b_k\theta^k)\a^2b^i+2\kappa(b_k\theta^k\b-b^2\theta)\b b^i\big\}+\f{\kappa'}{2(1-\kappa b^2)}\big\{(1-\kappa b^2)\b^2\cdot3\tau(1-b^2)b^i\nonumber\\
&&+\kappa\cdot3\tau(1-b^2)b^2\b^2 b^i\nonumber-6\tau(1-b^2)\b^2b^i\big\}\\
&=&(2\theta+\tau\b)y^i+\ta^2\theta^i-\f{1}{2(1-\kappa b^2)}\big\{(3\tau \kappa+2\tau-2\kappa\bk\theta^k)\ta^2
+3\tau[\kappa^2-\kappa+\kappa'(1-b^2)]\b^2\big\}b^i.
\end{eqnarray*}

We are not sure that it is absolute to make $\a$ dually flat before the whole process of deformations is finished. It is not priori. But if so, combining with Lemma \ref{beta2} and Lemma \ref{Rd} we can see that $\tG$ must be in the following form
$$\tG=Py^i+\ta^2(Q\theta^i+Rb^i).$$
Hence
\begin{eqnarray}\label{u}
\kappa^2-\kappa+\kappa'(1-b^2)=0
\end{eqnarray}
if $\tau\neq0$. In this case,
\begin{eqnarray}\label{tG}
\tG=(2\theta+\tau\b)y^i+\ta^2\theta^i-\f{1}{2(1-\kappa b^2)}(3\tau\kappa+2\tau-2\kappa\bk\theta^k)\ta^2b^i.
\end{eqnarray}

Carry out the second step of $\b$-deformations, then by (\ref{temp4}), (\ref{temp6}), (\ref{tG}) and Lemma \ref{beta2} we obtain
\begin{eqnarray*}
\hG&=&\tG+\rho'\left\{6\tau(1-b^2)\b y^i-\ta^2\left(3\tau(1-b^2)b^i+\f{\kappa}{1-\kappa b^2}\cdot3\tau(1-b^2)b^2b^i\right)\right\}\\
&=&\left\{2\theta+\tau[1+6\rho'(1-b^2)]\b\right\}y^i+\ta^2\theta^i
-\f{1}{2(1-\kappa b^2)}\left\{3\tau\kappa+2\tau+6\tau\rho'(1-b^2)-2\kappa b_k\theta^k\right\}\ta^2b^i.
\end{eqnarray*}
Let
\begin{eqnarray*}
\hat\theta=\theta+\f{1}{2}\tau[1+6\rho'(1-b^2)]\b.
\end{eqnarray*}
It is easy to verify that the inverse of $(\haij)$ is given by
\begin{eqnarray}\label{haIJ}
\hat a^{ij}=e^{-2\rho}\left(a^{ij}+\f{\kappa}{1-\kappa b^2}b^ib^j\right),
\end{eqnarray}
so
\begin{eqnarray*}
\hat\theta^i:=\hat a^{ij}\hat\theta_j=e^{-2\rho}\left\{\theta^i+\f{1}{2(1-\kappa b^2)}\left[2\kappa b_k\theta^k+\tau
+6\tau\rho'(1-b^2)\right]b^i\right\}.
\end{eqnarray*}
Hence $\hG$ can be reexpressed as
\begin{eqnarray*}
\hG=2\hat\theta y^i+\ha^2\hat\theta^i-\f{3\tau e^{-2\rho}}{2(1-\kappa b^2)}\left\{1+\kappa+4\rho'(1-b^2)\right\}\ha^2b^i.
\end{eqnarray*}
It is easy to see from the above equality that the sufficient condition for $\ha$ to be of dual flatness is
\begin{eqnarray}\label{rho}
1+\kappa+4\rho'(1-b^2)=0
\end{eqnarray}
if $\tau\neq0$. In this case, $\hG=2\hat\theta y^i+\ha^2\hat\theta^i$ where
\begin{eqnarray*}
\hat\theta=\theta-\f{1}{4}\tau(1+3\kappa)\b.
\end{eqnarray*}

So far, we have achieved our first aim. The output Riemannian metric $\ha$ will be dually flat  as long as the deformation factor $\kappa$ and $\rho$ satisfy (\ref{u}) and (\ref{rho}). In other words, the property of the Riemannian metric is clear. It is simple for our question, but the $1$-form is not enough.

Actually, under the deformations used above, we can see by Lemma \ref{beta2} and (\ref{u}) that
\begin{eqnarray*}
\trij&=&\f{1}{1-\kappa b^2}\left\{\rij+2\kappa b_k\theta^k\bi\bj-\kappa b^2(\theta_i\bj+\theta_j\bi)+3\tau \kappa'(1-b^2)b^2\bi\bj\right\}\\
&=&\theta_i\bj+\theta_j\bi+\f{1}{1-\kappa b^2}\big\{(3\tau+2\tau b^2-2b_k\theta^k)\aij
-[5\tau-2\kappa b_k\theta^k-3\tau \kappa'(1-b^2)b^2]\bi\bj\big\}\\
&=&\theta_i\bj+\theta_j\bi+\f{1}{1-\kappa b^2}\left\{3\tau+2\tau b^2-2\bk\theta^k\right\}\taij+\tau(3\kappa-5)\bi\bj\\
\tilde s_{ij}&=&\sij=\theta_i\bj-\theta_j\bi.
\end{eqnarray*}
Similarly, by Lemma \ref{beta3} and (\ref{rho}) we get
\begin{eqnarray*}
\hrij&=&\trij+\f{\kappa+1}{2(1-b^2)}\left\{6\tau(1-b^2)\bi\bj-\f{1}{1-\kappa b^2}\cdot3\tau(1-b^2)b^2\taij\right\}\\
&=&\theta_i\bj+\theta_j\bi+\f{e^{-2\rho}}{2(1-\kappa b^2)}\left\{6\tau+\tau b^2-3\tau \kappa b^2-4\bk\theta^k\right\}\haij+2\tau(3\kappa-1)\bi\bj,\\
\hat s_{ij}&=&\sij=\theta_i\bj-\theta_j\bi.
\end{eqnarray*}
If we use $\hat\theta$ instead of $\theta$ to express $\hrij$ and $\hat s_{ij}$, then
\begin{eqnarray*}
\hrij&=&\hat\theta_i\hbj+\hat\theta_j\hbi+\f{e^{-2\rho}}{2(1-\kappa b^2)}\left\{6\tau+\tau b^2-3\tau \kappa b^2-4\bk\theta^k\right\}\haij1+\f{3}{2}\tau(5\kappa-1)\hbi\hbj,\\
\hat s_{ij}&=&\hat\theta_i\hbj-\hat\theta_j\hbi,
\end{eqnarray*}
where $\hbi=\bi$ according to $\b$-deformations.

No matter which expression is chose, one can see that the covariant derivation of the corresponding $1$-forms share common features. Firstly, $\sij$ has some invariance. Secondly, $\rij$ always is composed at most by three terms, including the `linear' terms $\theta_i\bj+\theta_j\bi$ and $\aij$, and most important, the `nonlinear' term $\bi\bj$. More specifically, the linear combination of two $1$-forms satisfying
$$\rij=\theta_i\bj+\theta_j\bi+c(x)\aij,\qquad\sij=\theta_i\bj-\theta_j\bi$$
will have the same properties. We believe that it is the best form for $\b$ to be of such linear structure. Hence our second aim is to make $\b$ linear.

Carry our the third step of $\b$-deformations, then by (\ref{temp4}) and Lemma \ref{beta3} we obtain
\begin{eqnarray*}
\brij&=&\nu\hrij+\nu'\cdot6\tau(1-b^2)\bi\bj,\\
&=&\bar\theta_i\bbj+\bar\theta_j\bbi+\f{e^{-2\rho}\nu}{2(1-\kappa b^2)}\left\{6\tau+\tau b^2-3\tau \kappa b^2-4\bk\theta^k\right\}\baij
+\frac{3}{2}\tau\left\{(5\kappa-1)\nu+4(1-b^2)\nu'\right\}\hbi\hbj,\\
\bar s_{ij}&=&\nu\sij=\nu(\hat\theta_i\hbj-\hat\theta_j\hbi)=\bar\theta_i\bbj-\bar\theta_j\bbi,
\end{eqnarray*}
where $\bar\theta:=\hat\theta$. So the sufficient condition for $\bb$ to be linear is
\begin{eqnarray}\label{nu}
(5\kappa-1)\nu+4(1-b^2)\nu'=0,
\end{eqnarray}
if $\tau\neq0$.

In order to complete the deformations. We need to choose some suitable deformation factors. It is easy to see that $\kappa=1$ and $\kappa=0$ are both solutions of (\ref{u}), and they also satisfy (\ref{conditiononk}) since $b^2<1$ for Randers metrics.

If $\kappa=1$, then by (\ref{rho}) and (\ref{nu}) the deformation factors can be taken as
$$e^\rho=\sqrt{1-b^2},\qquad\nu=-(1-b^2),$$
which just corresponds to the navigation expression of Randers metrics. On can verify that
\begin{eqnarray*}
\brij=\bar\theta_i\bbj+\bar\theta_j\bbi-(2\bar b_k\bar\theta^k+3\tau)\baij,\qquad
\bar s_{ij}=\bar\theta_i\bbj-\bar\theta_j\bbi,
\end{eqnarray*}
where $\bar\theta^i:=\bar a^{ij}\bar\theta_j$. The above equalities are equivalent to
$$\bbij=2\bar\theta_i\bbj+c(x)\baij,$$
where $c(x)=-2\bar b_k\bar\theta^k-3\tau$. Replace $\ba$ and $\bb$ as $h$ and $W$, then the above discussion shows that the second item of Theorem \ref{main1} holds when $F=\a+\b$ is dually flat.

Conversely, the above deformations are reversible. In fact, according to the navigation expression of Randers metrics we know that
$$\a=(1-\bar b)^{-1}\sqrt{(1-\bar b)\ba^2+\bb^2},\quad\b=-(1-\bar b^2)^{-1}\bb.$$
Hence, if $\ba$ and $\bb$ , namely $h$ and $W$, satisfy the second item of Theorem \ref{main1}, then $\a$ and $\b$ will satisfy (\ref{G})-(\ref{sij}), thus $F=\a+\b$ is dually flat.

If $\kappa=0$, then by (\ref{rho}) and (\ref{nu}) the deformation factors can be taken as
$$e^\rho=(1-b^2)^\frac{1}{4},\qquad\nu=(1-b^2)^{-\frac{1}{4}}.$$
In this case, by (\ref{haIJ}) we have
$$\bar b^2=\nu b_i\nu b_je^{-2\rho}\left(a^{ij}+\f{\kappa}{1-\kappa b^2}b^ib^j\right)=\f{b^2}{1-b^2},$$
namely $(1+\bar b^2)(1-b^2)=1$. Hence
\begin{eqnarray}\label{fang}
\a=(1+\bar b^2)^\frac{1}{4}\ba,\qquad\b=(1+\bar b^2)^{-\frac{1}{4}}\bb.
\end{eqnarray}
In other words, the deformations are reversible. On can verify that
\begin{eqnarray*}
\brij=\bar\theta_i\bbj+\bar\theta_j\bbi-\left\{2\bar b_k\bar\theta^k-3\tau(1-b^2)^{-\frac{3}{4}}\right\}\baij,\qquad\bar s_{ij}=\bar\theta_i\bbj-\bar\theta_j\bbi,
\end{eqnarray*}
which is equivalent to
$$\bbij=2\bar\theta_i\bbj+c(x)\baij,$$
where $c(x)=-2\bar b_k\bar\theta^k+3\tau(1-b^2)^{-\frac{3}{4}}$. Hence, $F=\a+\b$ is dually flat if and only if the third item of Theorem \ref{main1} holds.

\section{Some constructions}\label{s4}

In this section, we aim to provide some explicit dually flat Randers metrics.

Firstly, by the arguments in Section 3 we can see that the $1$-forms which satisfy (\ref{yy}) play an important role in our question, which inspire us to introduce the following concept.
\begin{definition}
Let $\a$ be a locally dually flat Riemannian metric on a manifold $M$. Suppose that the spray coefficients $\G$ of $\a$ are given in an adapted coordinate system by (\ref{duallyflat}) with some $1$-form $\theta$ on $M$. Then an $1$-form $\b$ on $M$ is said to be {\em dually related} with respect to $\a$ if
\begin{eqnarray}\label{duallyrelated}
\bij=2\theta_i\bj+c(x)\aij,
\end{eqnarray}
where $c(x)$ is a scalar function on $M$.
\end{definition}

Using the above concept, Theorem \ref{main1} tell us that
\begin{thm}
A Randers metric $F=\a+\b$ is locally dually flat if and only if $h$ is locally dually flat and $W^\flat$ is dually related with respect to $h$, where $(h,W)$ is the navigation data of $F$.
\end{thm}

Before the further discussions, it is worth to point out that $\a$ and $\b$ satisfying
\begin{eqnarray}\label{special}
\G=2\theta y^i+\a^2\theta^i,\qquad\bij=2\theta_i\bj-2b_k\theta^k\aij
\end{eqnarray}
is a very special case. In our opinion, it is a trivial case in a sense.

The first reason is that this property is preserved by any $\b$-deformations. Specifically speaking, if a couple of data ($\a$,$\b$) has such property, then no matter what $\b$-deformations are carried out, the output ($\ba$,$\bb$) (including the middle outputs ($\ta$,$\tb$) and ($\ha$,$\hb$)) also has the same property. This phenomenon is not hard to be found in the arguments of Section 3.

Secondly, if $\a$ and $\b$ satisfy (\ref{special}), then for any suitable function $\phi(s)$, the $\ab$-metric $F=\a\p(\frac{\b}{\a})$ is dually flat. This result have been proved in \cite{xoldf}. Actually, it will be hold too for the general $\ab$-metric $F=\a\phi(b^2,\frac{\b}{\a})$.

We don't known how to solve (\ref{special}) yet. It is possible that there is not any non-zero $1$-form satisfying (\ref{special}) when $\a$ is non-Euclidean. Anyway, we will avoid this `trivial' case in the following discussions.

Let $\a$ and $\b$ be
\begin{eqnarray}
&\displaystyle\a=\frac{\sqrt{(1+\mu|x|^2)|y|^2-\mu\langle x,y\rangle^2}}{1+\mu|x|^2},&\label{csc}\\
&\displaystyle\b=\frac{\lambda\langle x,y\rangle+(1+\mu|x|^2)\langle a,y\rangle-\mu\langle a,x\rangle\langle x,y\rangle}{(1+\mu|x|^2)^\frac{3}{2}},\label{cc}&
\end{eqnarray}
where $\lambda$ is a constant number and $a$ is a constant vector, then $\a$ is of constant sectional curvature $\mu$ and
\begin{eqnarray}\label{cfactor}
\bij=\f{\lambda-\mu\langle a,x\rangle}{\sqrt{1+\mu|x|^2}}\aij,
\end{eqnarray}
which means that $\b$ is closed and {\it conformal} with respect to $\a$. These special Riemannian metrics and $1$-forms play an important role in projective Finsler geometry\cite{yct-dhfp,yct-zhm-onan}.

In terms of (\ref{csc}) and (\ref{cc}) we can construct some dually flat Riemannian metrics and dually related $1$-forms. The main method is still $\b$-deformations.

\begin{thm}
The Riemannian metrics
\begin{eqnarray}\label{dfR}
\ba=\frac{\sqrt{(1+\mu|x|^2)|y|^2-\mu\langle x,y\rangle^2}}{(1+\mu|x|^2)^\frac{3}{4}}
\end{eqnarray}
are dually flat on $\mathbb B^n(r_\mu)$, and the $1$-forms
\begin{eqnarray}\label{drb}
\bb=\f{\lambda\xy}{(1+\mu|x|^2)^\frac{5}{4}}
\end{eqnarray}
are dually related with respect to $\ba$.
\end{thm}
\begin{proof}
Direct computations show that
\begin{eqnarray}\label{gpp}
\G=Py^i,
\end{eqnarray}
where $P=-\frac{\mu\xy}{1+\mu\xx}$, and by (\ref{cfactor}) we have
\begin{eqnarray}\label{rs}
\rij=\sigma\aij,\qquad\sij=0,
\end{eqnarray}
where $\sigma=\frac{\lambda-\mu\langle a,x\rangle}{\sqrt{1+\mu|x|^2}}$.

Carry out the first step $\b$-deformations, then by (\ref{gpp}), (\ref{rs}) and Lemma \ref{beta1} we obtain
\begin{eqnarray*}
\tG&=&\G-\f{\kappa}{2(1-\kappa b^2)}\roo b^i+\f{\kappa'}{2(1-\kappa b^2)}\left\{(1-\kappa b^2)\b^2r^i+\kappa r\b^2b^i-2r_0\b b^i\right\}\\
&=&Py^i-\f{\sigma}{2(1-\kappa b^2)}(\kappa\a^2+\kappa'\b^2)b^i.
\end{eqnarray*}
Combining with Lemma \ref{beta2} one can see that $\ha$ cann't be dually flat only until $\tG$ has the following form
$$\tG=Py^i+Q\ta^2\b^i.$$
So $\kappa$ must satisfy the following equation
$$\kappa'=-\kappa^2.$$
It is obvious that $\kappa$ can be taken as $\kappa=0$, which means that the first step of $\b$-deformations is not necessary.

Carry out the second step of $\b$-deformations, then by (\ref{gpp}), (\ref{rs}) and Lemma \ref{beta2} we obtain
\begin{eqnarray}
\hG&=&\G+\rho'(2r_0y^i-\a^2r^i)
=(P+2\sigma\rho'\b)y^i-\sigma\rho'\a^2 b^i
=(P+2\sigma\rho'\b)y^i-\sigma\rho'e^{-2\rho}\ha^2 b^i.\label{hGtheta}
\end{eqnarray}
Hence, $\ha$ is dually flat if and only if
$$P+2\sigma\rho'\b=-2\sigma\rho'e^{-2\rho}b^i\hat y_i,$$
where $\hat y_i:=\haij y^j$. The above equality is equivalent to
\begin{eqnarray}\label{prho}
P=-4\sigma\rho'\b.
\end{eqnarray}
It is easy to see that the constant vector $a$ must be zero if $\mu\neq0$. In this case, $\b$ is given by
\begin{eqnarray*}
\b=\f{\lambda\xy}{(1+\mu|x|^2)^\frac{3}{2}}.
\end{eqnarray*}
Obviously, it is trivial when $\lambda=0$, so we will assume $\lambda\neq0$ in the following arguments. Direct computations show that
$$b^2=\f{\lambda^2\xx}{1+\mu\xx},$$
namely,
\begin{eqnarray}\label{bx}
(\lambda^2-\mu b^2)(1+\mu\xx)=\lambda^2.
\end{eqnarray}
Thus, the equation (\ref{prho}) becomes
$$\rho'=\f{\mu}{4(\lambda^2-\mu b^2)},$$
so $\rho$ can be chose as
$$\rho=\f{1}{4}\left(\ln\lambda^2-\ln(\lambda^2-\mu b^2)\right)=\f{1}{4}\ln(1+\mu\xx),$$
which implies that (\ref{dfR}) are dually flat. In this case, the corresponding $1$-forms $\hat\theta$ are given by
\begin{eqnarray*}
\hat\theta=\f{1}{2}(P+2\sigma\rho'\b)=-\sigma\rho'\b=-\f{\mu\xy}{4(1+\mu\xx)}
\end{eqnarray*}
according to (\ref{hGtheta}), (\ref{prho}) and (\ref{bx}). Moreover, by Lemma \ref{beta2} we have
\begin{eqnarray*}
\hrij=\rij-2\rho'(\bi\rj+\bj\ri-r\aij)=\sigma(1+2b^2\rho')e^{-2\rho}\haij-4\sigma\rho'\bi\bj.
\end{eqnarray*}

Carry out the third step of $\b$-deformations, then
\begin{eqnarray*}
\brij&=&\nu\hrij+\nu'(\bi\rj+\bj\ri)\\
&=&\nu\sigma(1+2b^2\rho')e^{-2\rho}\haij+2(\sigma\nu'-4\sigma\rho'\nu)\bi\bj\nonumber\\
&=&\nu\sigma(1+2b^2\rho')e^{-2\rho}\baij+2\left(2-\f{\nu'}{\nu\rho'}\right)\bar\theta_i\bbj,
\end{eqnarray*}
where $\bar\theta:=\hat\theta$. The above equality is equivalent to
\begin{eqnarray*}
\bbij=\nu\sigma(1+2b^2\rho')e^{-2\rho}\baij+2\left(2-\f{\nu'}{\nu\rho'}\right)\bar\theta_i\bbj
\end{eqnarray*}
since $\bar s_{ij}=0$. It is obviously that $\bb$ is dually related with respect to $\ba$ if and only if
$$\f{\nu'}{\nu}=\rho',$$
thus $\nu$ can be taken as
$$\nu=e^\rho=(1+\mu\xx)^\frac{1}{4},$$
in this case, $\bb$ is given by (\ref{drb}), and
$$\bbij=\bar c(x)\baij+2\bar\theta_i\bbj,$$
where
$$\bar c(x)=\frac{\lambda}{2}\cdot\f{2+\mu\xx}{(1+\mu\xx)^\frac{3}{4}}.$$

Finally, one can verify that
$$\bar c(x)+2\bar b_k\bar\theta^k=\f{\lambda}{(1+\mu\xx)^\frac{3}{4}},$$
which means that the $1$-forms (\ref{drb}) are non-trivial when $\lambda\neq0$.
\end{proof}

\begin{proof}[Proof of Theorem \ref{main2}]
According the second item of Theorem \ref{main1} and combining with (\ref{fang}), the Randers metrics
$$F=(1+\bar b^2)^{\frac{1}{4}}\ba+(1+\bar b^2)^{-\frac{1}{4}}\bb$$
are dually flat if $\ba$ and $\b$ are given by (\ref{dfR}) and (\ref{drb}), where $\bar b^2$ is given by
$$\bar b^2=\f{\lambda^2\xx}{1+\mu\xx}.$$
Thus $F$ are given by (\ref{example}).

Similarly, according the third item of Theorem \ref{main1} we will obtain the following dually flat Randers metrics
\begin{eqnarray*}
F=\f{\sqrt[4]{1+\mu\xx}\sqrt{(1+(\mu-\lambda^2)\xx)\yy-(\mu-\lambda^2)\xy^2}}
{1+(\mu-\lambda^2)\xx}-\f{\lambda\xy}{(1+(\mu-\lambda^2)\xx)\sqrt[4]{1+\mu\xx}},
\end{eqnarray*}
which are the same as (\ref{example}), just by changing $\mu$ to $\mu+\lambda^2$ and $\lambda$ to $-\lambda$.
\end{proof}

\noindent Changtao Yu\\
School of Mathematical Sciences, South China Normal
University, Guangzhou, 510631, P.R. China\\
aizhenli@gmail.com

\end{document}